\documentclass{amsart}
\usepackage{amsfonts,amssymb}
\usepackage{enumerate,graphicx}
\usepackage{caption}
\usepackage{subcaption}
\usepackage[square,numbers]{natbib}
\usepackage{mathrsfs,bbm}
\usepackage{tikz,tikzscale,tikz-cd}
\usepackage{multirow,xparse,fp}
\usepackage{environ}
\usepackage{amstext}
\usepackage{hyperref}
\usepackage{diagbox}
 \usepackage[pagewise]{lineno}
\usepackage{comment}
\newtheorem{theorem}{Theorem}[section]

\newtheorem{corollary}[theorem]{Corollary}
\newtheorem{lemma}[theorem]{Lemma}
\theoremstyle{definition}

\newtheorem{definition}[theorem]{Definition}
\newtheorem{example}[theorem]{Example}

\newtheorem{remark}[theorem]{Remark}

\begin{document}
\title{Thermodynamic formalism and large deviation principle of multiplicative Ising models}

\author[Jung-Chao Ban]{Jung-Chao Ban}
\address[Jung-Chao Ban]{Department of Mathematical Sciences, National Chengchi University, Taipei 11605, Taiwan, ROC.}
\address{Math. Division, National Center for Theoretical Science, National Taiwan University, Taipei 10617, Taiwan. ROC.}
\email{jcban@nccu.edu.tw}

\author[Wen-Guei Hu]{Wen-Guei Hu}
\address[Wen-Guei Hu]{College of Mathematics, Sichuan University, Chengdu, 610064, China}
\email{wghu@scu.edu.cn}

\author[Guan-Yu Lai]{Guan-Yu Lai}
\address[Guan-Yu Lai]{Department of Applied Mathematics, National Yang Ming Chiao Tung University, Hsinchu 30010, Taiwan, ROC.}
\email{guanyu.am04g@g2.nctu.edu.tw}


\thanks{Ban is partially supported by the Ministry of Science and Technology, ROC (Contract MOST 109-2115-M-004-002-MY2 and 108-2115-M-004-003). Hu is partially supported by the National Natural Science Foundation of China (Grant 11601355).}

\date{}

\baselineskip=1.2\baselineskip

\begin{abstract}
The aim of this study is tree-fold. First, we investigate the thermodynamics of the Ising models with respect to 2-multiple Hamiltonians. This extends the previous results of [Chazotte and Redig, Electron. J. Probably., 2014] to $\mathbb{N}^d$. Second, we establish the large deviation principle (LDP) of the average $\frac{1}{N} S_N^G$, where $S_N^G$ is a 2-multiple sum along a semigroup generated by k numbers which are k co-primes. This extends the previous results [Ban et al. Indag. Math., 2021] to a board class of the long-range interactions. Finally, the results described above are generalized to the multidimensional lattice $\mathbb{N}^d, d\geq1$. 
\end{abstract}
\maketitle


\section{Introduction}

In this article, we investigate the thermodynamic formalism (e.g., Gibbs measures, entropy and pressure functions) and the large deviation principle of the multiplicative Ising models on $\mathbb{N}^{d}, d\geq 1$. Before presenting the main results, we provide the motives behind this study. Consider the lattice spin systems with Ising $\pm 1$ spins on $\mathbb{N}$. The Hamiltonian of the standard Ising model on each configuration $\sigma \in\{-1,+1\}^{\mathbb{N}}$ is  
\[
H(\sigma )=-\beta \left( \sum_{i\in \mathbb{N}}J\sigma _{i}\sigma
_{i+1}+h\sum_{i\in \mathbb{N}}\sigma _{i}\right) ,
\]
where the parameter $\beta $ is the inverse temperature, $J$ is the coupling strength and $h$ stands for the magnetic field. Note that the Hamiltonian $H(\sigma )$ is a nearest neighbor translation-invariant interaction. The thermodynamics of the Ising models, e.g., the existence of the Gibbs measure, phase transition problem, entropy and pressure functions have been studied in depth. The best reference on this topic is \cite{georgii2011gibbs}. 

Motivated from multiple ergodic theory (we refer the reader to \cite{fan2014some, frantzikinakis2011some} for detailed definitions and recent advances in multiple ergodic theory) and the previous works of Kifer \cite{kifer2010nonconventional}, Kifer and Varadhan \cite{kifer2014nonconventional}, Fan, Liao and Ma \cite{fan2012level} on the nonconventional averages, Carinci et al. \cite{carinci2012nonconventional} and Chazottes and Redig \cite{chazottes2014thermodynamic} investigate the `multiplicative Ising models'. That is, they consider the Hamiltonian on each configuration $\sigma \in \{-1,+1\}^{\mathbb{N}}$ as 
\[
H^{[m]}(\sigma )=-\beta \left( \sum_{i\in \mathbb{N}}J\sigma _{i}\sigma
_{2i}+h\sum_{i\in \mathbb{N}}\sigma _{i}\right) .
\]
It should be noted that the Hamiltonian $H^{[m]}(\sigma )$ is a long-range
non-translation invariant interaction. For $h=0$, the existence and
uniqueness of the Gibbs measure with respect to $H^{[m]}(\sigma )$ is
constructed \cite{chazottes2014thermodynamic}, which is multiplication invariant. 
\begin{theorem}[Theorem 3.2 \cite{chazottes2014thermodynamic}]\label{thm 1.1}
	Let 
	\begin{equation*}
		\mu_{N}(\sigma_{\mathcal{N}_{2N}})=\frac{e^{-H_{N}(\sigma_{\mathcal{N}_{2N}})}}{\sum_{{\sigma_i}=\pm 1, i\in \mathcal{N}_{2N}}e^{-H_{N}(\sigma_{\mathcal{N}_{2N}})} }\footnote{ $\mathcal{N}_{2N}=\{ i\in\mathbb{N} : i\leq 2N \}$  }
	\end{equation*}
	be the finite-volume probability measure corresponding to the Hamiltonian 
	\begin{equation*}
		H_{N}(\sigma_{\mathcal{N}_{2N}})= -\beta\sum_{i=1}^{N}  \sigma_{i} \sigma_{2i}.
	\end{equation*}
	Then
	\item[1.]Unique limit measure : The measure $\mu_{N}$ have a unique weak limit denoted by $\mu_{\infty}$ which is the Gibbs measure.
	\item[2.]Independent Ising layers : Under $\mu_{\infty}$, the $\{\tau_{\ell}^{i}=\sigma_{i2^\ell}\}_{\ell=0}^{\infty} , i\in 2\mathbb{N}_0+1$ are independent and distributed according to the standard Ising model measure $\mu_\infty^{Ising}$ with a free boundary condition on the left.
	\item[3.]Multiplication invariance : The measure $\mu_{\infty}$ is multiplication invariant, i.e., for all $i\in\mathbb{N}$, $\sigma=(\sigma_j)_{j=1}^{\infty}$ and $T_i\sigma=(\sigma_{ij})_{j=1}^{\infty}$ have the same distribution.
	
\end{theorem}
Carinci et al. \cite{carinci2012nonconventional} established the large deviation principle of the average $\frac{1}{N}S_{N}(\sigma ):=\frac{1%
}{N}\sum_{i=1}^{N}\sigma _{i}\sigma _{2i}$ by calculating the rigorous
formula of the free energy function 
\[
F_{r}(\beta )=F_{\mathbb{P}_{r}}(\beta
)=\lim_{N\rightarrow \infty }\frac{1}{N}\log \mathbb{E}_{\mathbb{P}
	_{r}}\left(e^{\beta S_{N}}\right),
\]
where $\mathbb{P}_{r}$ is the product of Bernoulli measures with the parameter $r$ on two symbols $\{-1,+1\}$. 
\begin{theorem}[Theorem 4.1 \cite{carinci2012nonconventional}]\label{thm 1.2}
	The explicit expression of the free energy function associated to the multiple sum $S_N$ is  
	\begin{equation*}
		\begin{aligned}
			F_r(\beta)=\log\left( (r(1-r))^{\frac{3}{4}} |v^T\cdot e_+| \Lambda_+\right) +\mathcal{G}(\beta),
		\end{aligned}
	\end{equation*}
	where 	
	\begin{equation*}
		\begin{aligned}
			\mathcal{G}(\beta)=\frac{1}{2}\sum_{\ell=1}^{\infty}\frac{1}{2^{\ell}}\log\left(1+\left( \frac{2\cosh(h)}{|v^T\cdot e_+|^2}-1 \right)\left(\frac{\Lambda_-}{\Lambda_+} \right)^\ell\right)
		\end{aligned}
	\end{equation*}	
	and $\Lambda_{\pm}=e^\beta \left( \cosh(h)\pm \sqrt{\sinh^2(h)+e^{-4\beta}} \right)$, $v^T=(e^{h/2},e^{-h/2})$, $h=\frac{1}{2}\log(r/(1-r))$, $e_+=\frac{w_+}{\| w_+ \|}$ with $w_+=(-e^{-\beta}, e^{h+\beta}-\Lambda_+)^T$.
\end{theorem}
Theorem \ref{thm 1.2} has been recently generalized to multidimensional
lattice $\mathbb{N}^{d}, d\geq 1$ \cite{ban2021LDP}.
\begin{theorem}[Theorem 3.2 \cite{ban2021LDP}]\label{thm 1.3}
	For any $d\geq 1$ and ${\bf p}=(p_1,...,p_d)\in \mathbb{N}^d$, the following statements hold true.
	\item[1.]The explicit expression of the free energy function associated to the multiple sum\footnote{${\bf i}\cdot {\bf p}=(i_1p_1,...,i_d p_d)$ and $\mathcal{N}_{N_1\times \cdots \times N_d}=\{ {\bf i}=(i_1,...,i_d)\in\mathbb{N}^d : i_j \leq N_j,\forall 1\leq j \leq d \}$}
	\[
	S_{N_1\times \cdots \times N_d}^{{\bf p}}=\frac{1}{N_1\cdots N_d}\sum_{{\bf i}\in \mathcal{N}_{N_1\times \cdots \times N_d} }\sigma_{\bf i} \sigma_{{\bf i}\cdot {\bf p}}
	\] 
	is
	\begin{equation*}\label{F_r(beta)}
		\begin{aligned}
			F_r(\beta)
			=\frac{2p_1\cdots p_d-1}{2p_1\cdots p_d}\log (r(1-r))+\frac{p_1\cdots p_d-1}{p_1\cdots p_d}\log|v^T\cdot e_+|^2 +\log\Lambda_++\mathcal{G}(\beta),
		\end{aligned}
	\end{equation*}
	where $\Lambda_{\pm}$, $v^T$, $h$, $e_+$ are defined as above and 
	\begin{equation*}
		\begin{aligned}
			\mathcal{G}(\beta)=\sum_{\ell=1}^{\infty}\frac{(p_1\cdots p_d-1)^2}{(p_1\cdots p_d)^{\ell+1}}\log\left(1+\left( \frac{2\cosh(h)}{|v^T\cdot e_+|^2}-1 \right)\left(\frac{\Lambda_-}{\Lambda_+} \right)^\ell\right).
		\end{aligned}
	\end{equation*}
	\item[2.]The function $F_r(\beta)$ is differentiable with respect to $\beta \in \mathbb{R}$.
	\item[3.]The multiple average satisfies a LDP with rate function given by
	\begin{equation*}
		\begin{aligned}
			I_{r}(x)=&\sup_{\beta \in \mathbb{R}}\left( \beta x-F_r(\beta) \right).
		\end{aligned}
	\end{equation*}
	Furthermore, if $(F_r)'(\eta)=y$, then $I_{r}(y)=\eta y- F_r(\eta)$. 
\end{theorem}
 The aim of this paper is to study the multiplcative Ising models on $%
\mathbb{N}^{d}$ and the objective is three-fold.

\textbf{1}. The thermodynamics of the Ising model with respect to 2-multiple Hamiltonian: We establish the thermodynamics formalsm of the multiplicative Ising models on multidimensional lattice $\mathbb{N}^{d}$ with respect to the Hamiltonian $\sum_{i\in \mathbb{N}}\sigma _{i}\sigma _{2i}$ (call it \emph{2-multiple Hamiltonian}). Theorem \ref{thm mu 2-multiple} extends Theorem \ref{thm 1.1} to $\mathbb{N}^{d}$ and the formula of the Kolmogorov-Sinai (KS) entropy with respect to the limiting measure constructed in Theorem \ref{thm mu 2-multiple} is presented in Theorem \ref{thm KS 2multi Nd}.

\textbf{2}. The large deviation principle of the average $\frac{1}{N}S_{N}^{G}$: Let $k\geq 1$ and $p_{1},\ldots ,p_{k}$ be co-primes, and $G=\langle p_{1},...,p_{k} \rangle=\{1=\ell_{1}<\ell_{2}<\cdots\} $ be a semigroup generated by $p_{1},...,p_{k}$. The rigorous formula for the free energy function and
the large deviation principle with respect to the average $S_{N}^{G}(\sigma):=\sum_{i=1}^{N}\sigma _{\ell_{j(i)}}\sigma _{\ell_{j(i)+1}}$ on $\mathbb{N}$ is given in Theorem \ref{thm main 1d 2.5multiple}. Clearly, $2$-multiple Hamiltonian is a special case of the $S_{N}^{G}$ with $G=\langle 2\rangle$. Therefore, Theorem \ref{thm main 1d 2.5multiple} can be seen as a generalization of Theorem \ref{thm 1.2} to a board class of the long-range non-translation invariant interactions. 

\textbf{3}. The $\mathbb{N}^d$ version of the aforementioned results: The $\mathbb{N}^d$ version of Theorem \ref{thm main 1d 2.5multiple} (or the generalization of Theorem \ref{thm 1.3}) are presented in Theorems \ref{thm general} and \ref{theorem main Nd 2.5multiple}. We remark that Theorem \ref{thm general} provides a general formula for $F_r(\beta)$. Meanwhile, Theorem \ref{theorem main Nd 2.5multiple} presents a explicit formula for $F_r(\beta)$ if the sum of $F_r(\beta)$ is along some specific direction. In this circumstance, we call $F_r(\beta)$ the {\bf directional free energy function} which is defined in Section \ref{section 4.2}. We emphasize that the calculation of the rigorous formula for $F_r(\beta)$ along other directions is extremely difficult since the independent sublattices according to the constrains ${\bf G}$ are quite hard to analysis. The notation ${\bf G}$ represents the semigroup generated by the vectors ${\bf p}_1,...,{\bf p}_k \in \mathbb{N}^d$, the formal definition can be found in Section \ref{section 2}. And the generlizations of Theorems \ref{thm mu 2-multiple} and \ref{thm KS 2multi Nd} are given in Theorems \ref{thm mu 2.5-multiple} and \ref{thm Ks 2.5 Nd} respectively.

We remark that Peres et al. \cite{peres2014dimensions} computed the
Hausdorff and Minkowski dimensions of the set 
\[
X_{\Omega }^{G}=\{(x_k)_{k=1}^{\infty}\in \Sigma _{m}:x|_{iG}\in \Omega \text{ for all }i\text{, }\gcd(i,G)=1\}\text{,}
\]
where $\Omega \subseteq \Sigma _{m}:=\{0,\ldots ,m-1\}^{\mathbb{N}}$ and $x|_{iG}:=(x_{i\ell_{k}})_{k=1}^{\infty }$. The work of \cite{peres2014dimensions} concerns the decomposition of the lattices $\mathbb{N}$ into independent sublattices according to the multiple constraints $G$ and calculates its density among the entire lattice $\mathbb{N}$. The study can be regarded as a study of the thermodynamic formalism of $X_\Omega^{G}$ with the potential function $S_N^{G}$ in $\mathbb{N}^d, d\geq1$. In addition, the challenge of the current problem is to look at the same issue as \cite{peres2014dimensions} in $\mathbb{N}^d, d\geq1$. The results can be summarized in the following table.
\begin{center}
	\begin{tabular}{|l|l|l|}
		\hline
		\multicolumn{3}{|c|}{Thermodynamics and LDP}\\
		\hline
		\multirow{7}{*}{Limit measure and KS entropy}&2-multiple, $d=1$&Theorem \ref{thm 1.1}\\
		\cline{2-3}
		&\multirow{3}{*}{2-multiple, $d\geq1$}&Theorem \ref{thm iid 2-multiple}\\
		\cline{3-3}
		&&Theorem \ref{thm mu 2-multiple}\\
		\cline{3-3}
		&&Theorem \ref{thm KS 2multi Nd}\\
		\cline{2-3}
		&\multirow{3}{*}{$S^{\bf G}_{N_1\times\cdots \times N_d},d\geq 1$}&Theorem \ref{thm iid 2.5-multiple}\\
		\cline{3-3}
		&&Theorem \ref{thm mu 2.5-multiple}\\
		\cline{3-3}
		&&Theorem \ref{thm Ks 2.5 Nd}\\
		\hline
		\multirow{4}{*}{Formula of the free energy function}&2-multiple, $d=1$&Theorem \ref{thm 1.2}\\
		\cline{2-3}
		&2-multiple, $d\geq1$&Theorem \ref{thm 1.3}\\
		\cline{2-3}
		&$S^G_N,d=1$&Theorem \ref{thm main 1d 2.5multiple}\\
		\cline{2-3}
		&$S^{\bf G}_{N_1\times\cdots \times N_d},d\geq 1$&Theorem \ref{thm general} and \ref{theorem main Nd 2.5multiple}\\
		\hline
	\end{tabular}
\end{center}



\section{Preliminaries}\label{section 2}
In this section, we provide necessary materials and results on the decomposition of the multidimensional lattice $\mathbb{N}^d$ into independent sublattices and calculate their densities.

Let ${\bf p}_1,...,{\bf p}_k\in \mathbb{N}^d$ with ${\bf p}_i=(p_{i1},...,p_{id})$ for all $1\leq i \leq k$. Denote by $\mathcal{I}_{{\bf p}_1,...,{\bf p}_k}=\{ (i_1,...,i_d)\in \mathbb{N}^d : p_{j1}\nmid i_1   \mbox{ \rm or }...\mbox{ \rm  or }p_{jd}\nmid i_d \mbox{ for all } 1\leq j \leq k  \}$, $\mathcal{M}_{{\bf p}_1,...,{\bf p}_k}=\{ (p_{11}^{\ell_1}\cdots p_{k1}^{\ell_k},..., p_{1d}^{\ell_1} \cdots p_{kd}^{\ell_k}) : \ell_j \geq 0 \mbox{ for all } 1\leq j \leq k  \}$ and
\[
\mathcal{M}_{{\bf p}_1,...,{\bf p}_k}({\bf i})=\{ (i_1 p_{11}^{\ell_1}\cdots p_{k1}^{\ell_k},..., i_d p_{1d}^{\ell_1} \cdots p_{kd}^{\ell_k}) : \ell_j \geq 0 \mbox{ for all } 1\leq j \leq k  \}
\]
is the lattice $\mathcal{M}_{{\bf p}_1,...,{\bf p}_k}$ that starts at ${\bf i}$. We also need the following definitions.
\begin{definition}\label{definition:2.5}
	For $N_1,...,N_d$ and $M_1,...,M_d \geq 1$, let
	\item[1.]$S^{(j)}=\langle p_{1j},...,p_{kj}\rangle=\{ 1=\ell^{(j)}_1< \ell^{(j)}_2 < \cdots   \}$ for all $1\leq j \leq d$.
	\item[2.]$\mathcal{J}_{N_1\times \cdots \times N_d;M_1,...,M_d}=\{  
	{\bf i}\in \mathcal{N}_{N_1\times \cdots \times N_d} :  i_j \ell^{(j)}_{M_j} \leq N_j < i_j \ell^{(j)}_{M_j+1 }\mbox{ for all } 1\leq j \leq d\}$ be the subset of $\mathcal{N}_{N_1\times \cdots \times N_d}$ which satisfies $i_j \ell^{(j)}_{M_j} \leq N_j < i_j \ell^{(j)}_{M_j+1 }$ for all $1\leq j \leq d$.
	\item[3.]$\mathcal{K}_{N_1\times \cdots \times N_d;M_1,...,M_d}=\mathcal{J}_{N_1\times \cdots \times N_d;M_1,...,M_d}\cap~ \mathcal{I}_{{\bf p}_1,...,{\bf p}_k}$ be the subset of $\mathcal{J}_{N_1\times \cdots \times N_d;M_1,...,M_d}$ which belongs to $\mathcal{I}_{{\bf p}_1,...,{\bf p}_k}$.
\end{definition}
The following lemmas are $\mathbb{N}^d$ version of Lemmas 2.4 and 2.7 in \cite{ban2021entropy}.

\begin{lemma}\label{lemma Nd 2.4}
	For ${\bf p}_1,...,{\bf p}_k \in \mathbb{N}^d$ with $\gcd(p_{is},p_{js})=1$ for all $1\leq i \neq j \leq k$ and $1\leq s \leq d$,
	\begin{equation*}  
		\mathbb{N}^d=\bigsqcup_{{\bf i} \in \mathcal{I}_{{\bf p}_1,...,{\bf p}_k}}\mathcal{M}_{{\bf p}_1,...,{\bf p}_k}({\bf i}).
	\end{equation*}  
\end{lemma}

\begin{proof}
	We first claim that for all ${\bf i} \neq {\bf i}' \in \mathcal{I}_{{\bf p}_1,...,{\bf p}_k}$, $\mathcal{M}_{{\bf p}_1,...,{\bf p}_k}({\bf i})\cap\mathcal{M}_{{\bf p}_1,...,{\bf p}_k}({\bf i}')=\emptyset$. If we suppose not, then there exist ${\bf i} \neq {\bf i}' \in \mathcal{I}_{{\bf p}_1,...,{\bf p}_k}$ such that $\mathcal{M}_{{\bf p}_1,...,{\bf p}_k}({\bf i})\cap\mathcal{M}_{{\bf p}_1,...,{\bf p}_k}({\bf i}')\neq\emptyset$. Since ${\bf i} \neq {\bf i}'$, 
	\[
	(i_1 p_{11}^{\ell_1}\cdots p_{k1}^{\ell_k},..., i_d p_{1d}^{\ell_1} \cdots p_{kd}^{\ell_k})=(i'_1 p_{11}^{\ell'_1}\cdots p_{k1}^{\ell'_k},..., i'_d p_{1d}^{\ell'_1} \cdots p_{kd}^{\ell'_k})
	\]
	for some $\ell_i\neq \ell'_i, 1\leq i \leq k$. Without loss of generality, we may assume $\ell_i>\ell'_i$. Then by the $i_sp_{is}^{\ell_i-\ell'_i}\prod_{j\neq i} p_{js}^{\ell_j} = i'_s \prod_{j\neq i} p_{js}^{\ell'_j}$ for all $1\leq s \leq d$ and $\gcd(p_{is},p_{js})=1$ for all $1\leq i \neq j \leq k$ and $1\leq s \leq d$, we have $p_{is} | i'_s$ for all $1\leq s \leq d$. This contradicts with ${\bf i}'  \in \mathcal{I}_{{\bf p}_1,...,{\bf p}_k}$.
	
	It remains to show that the equality holds. For ${\bf i} \in \mathbb{N}^d$, there exist $\ell_{js}\geq 0$ for all $1\leq s \leq d$ and $1\leq j \leq k$ such that
	\[
	(i_1,...,i_d)=\left(i'_1 \prod_{j=1}^k p_{j1}^{\ell_{j1}},...,i'_d\prod_{j=1}^k p_{jd}^{\ell_{jd}}\right),
	\]
	where $p_{js}\nmid i'_s$ for all $1\leq j \leq k$ and for all $1\leq s \leq d$. Take $\ell_j=\min \{ \ell_{js} : 1\leq s \leq d \}$ for all $1\leq j \leq k$. Then, we have $(i_1,...,i_d) \in \mathcal{M}_{{\bf p}_1,...,{\bf p}_k}(\frac{i_1}{\prod_{j=1}^k p_{j1}^{\ell_{j}}},...,\frac{i_d}{\prod_{j=1}^k p_{jd}^{\ell_{j}}})$ and $(\frac{i_1}{\prod_{j=1}^k p_{j1}^{\ell_{j}}},...,\frac{i_d}{\prod_{j=1}^k p_{jd}^{\ell_{j}}})\in \mathcal{I}_{{\bf p}_1,...,{\bf p}_k}$. The converse is then clear.
\end{proof}

\begin{lemma}\label{lemma Nd decomposition 2.5 multiple}
	For $M_i,N_i \geq 1$ for all $1\leq i \leq d$, we have the following assertions.
	\item[1.]$\displaystyle|\mathcal{J}_{N_1\times \cdots \times N_d;M_1,...,M_d}|=\prod_{i=1}^d\left(\left\lfloor\frac{N_i}{\ell^{(i)}_{M_i}}\right\rfloor-\left\lfloor\frac{N_i}{\ell^{(i)}_{M_i+1}}\right\rfloor\right)$.
	\item[2.]$\displaystyle{\lim_{N_1,...,N_d\to \infty}}\frac{|\mathcal{K}_{N_1\times \cdots \times N_d;M_1,...,M_d}|}{|\mathcal{J}_{N_1\times \cdots \times N_d;M_1,...,M_d} |}=\prod_{i=1}^k\left(1-\frac{1}{p_{i1}p_{i2}\cdots p_{id}}\right)$.
	\item[3.]$\displaystyle{\lim_{N_1,...,N_d\to \infty}}\frac{1}{N_1\cdots N_d}\sum_{M_1=1}^{N_1}\cdots\sum_{M_d=1}^{N_d}|\mathcal{K}_{N_1\times \cdots \times N_d;M_1,...,M_d}|\log a_{M_1,...,M_d}=\\
	\sum_{M_1,...,M_d=1}^{\infty}\displaystyle{\lim_{N_1,...,N_d\to \infty}}\frac{1}{N_1\cdots N_d}|\mathcal{K}_{N_1\times \cdots \times N_d;M_1,...,M_d}|\log a_{M_1,...,M_d}$.
\end{lemma}

\begin{proof}
	\item[1.]Since ${\bf i}\in \mathcal{J}_{N_1\times \cdots \times N_d;M_1,...,M_d}$ if and only if $i_j \ell^{(j)}_{M_j} \leq N_j < i_j \ell^{(j)}_{M_j+1 }$ for all $1\leq j \leq d$. It follows that $ \frac{N_j}{\ell^{(j)}_{M_j+1 }} < i_j \leq \frac{N_j}{\ell^{(j)}_{M_j}} $ for all $1\leq j \leq d$. Therefore
	\[
	|\mathcal{J}_{N_1\times \cdots \times N_d;M_1,...,M_d}|=\prod_{i=1}^d\left(\left\lfloor\frac{N_i}{\ell^{(i)}_{M_i}}\right\rfloor-\left\lfloor\frac{N_i}{\ell^{(i)}_{M_i+1}}\right\rfloor\right).
	\]
	\item[2.]Let the complement of $\mathcal{I}_{{\bf p}_1,...,{\bf p}_k}$ be 
	\[
	\mathcal{I}_{{\bf p}_1,...,{\bf p}_k}^c=\bigcup_{j=1}^k \mathcal{S}_j=\bigcup_{j=1}^k \{ {\bf i} : p_{js} \mid i_s \mbox{ for all } 1\leq s \leq d \}.
	\]
	Since $\gcd(p_{is},p_{js})=1$ for all $1\leq i \neq j \leq k$ and $1\leq s \leq d$, we have for any $1\leq \ell \leq k$,
	\[
	\bigcap_{w=1}^{\ell} \mathcal{S}_{j_w}=\left\{ {\bf i} : \prod_{w=1}^{\ell}p_{j_w s} \mid i_s \mbox{ for all }  1\leq s \leq d \right\}.
	\]
	Then the inclusion–exclusion principle infers that
	\begin{equation*}
		\begin{aligned}
			\lim_{N_1,...,N_d\to \infty}\frac{|\mathcal{K}_{N_1\times \cdots \times N_d;M_1,...,M_d}|}{|\mathcal{J}_{N_1\times \cdots \times N_d;M_1,...,M_d} |}&=\lim_{N_1,...,N_d\to \infty}\frac{|\mathcal{J}_{N_1\times \cdots \times N_d;M_1,...,M_d}\cap \mathcal{I}_{{\bf p}_1,...,{\bf p}_k}|}{|\mathcal{J}_{N_1\times \cdots \times N_d;M_1,...,M_d} |}\\
			&=\lim_{N_1,...,N_d\to \infty}1-\frac{|\mathcal{J}_{N_1\times \cdots \times N_d;M_1,...,M_d}\cap \mathcal{I}^c_{{\bf p}_1,...,{\bf p}_k}|}{|\mathcal{J}_{N_1\times \cdots \times N_d;M_1,...,M_d} |}\\
			&=\lim_{N_1,...,N_d\to \infty}1-\left( \sum_{n=1}^{k}\frac{|  \mathcal{J}_{N_1\times \cdots \times N_d;M_1,...,M_d}\cap \mathcal{S}_{n}  |}{|\mathcal{J}_{N_1\times \cdots \times N_d;M_1,...,M_d}|}  \right.\\
			& -\sum_{1\leq n_1\neq n_2 \leq k}\frac{|  \mathcal{J}_{N_1\times \cdots \times N_d;M_1,...,M_d}\cap \mathcal{S}_{n_1}\cap \mathcal{S}_{n_2}  |}{|\mathcal{J}_{N_1\times \cdots \times N_d;M_1,...,M_d}|} \\
			&+\sum_{1\leq n_1\neq n_2 \neq n_3 \leq k} \frac{|  \mathcal{J}_{N_1\times \cdots \times N_d;M_1,...,M_d}\cap \mathcal{S}_{n_1}\cap \mathcal{S}_{n_2}\cap \mathcal{S}_{n_3}  |}{|\mathcal{J}_{N_1\times \cdots \times N_d;M_1,...,M_d}|} \\
			&\left.-\cdots+(-1)^{k-1} \frac{|  \mathcal{J}_{N_1\times \cdots \times N_d;M_1,...,M_d}\cap \mathcal{S}_{1}\cap \cdots \cap \mathcal{S}_{k}  |}{|\mathcal{J}_{N_1\times \cdots \times N_d;M_1,...,M_d}|} \right)\\
			&=\prod_{i=1}^k\left(1-\frac{1}{p_{i1}p_{i2}\cdots p_{id}}\right).
		\end{aligned}
	\end{equation*}
	\item[3.]Define
	\begin{equation*}
		\hat{K}_{N_1\times \cdots \times N_d;M_1,...,M_d}=
		\left\{
		\begin{array}{ll}
			|\mathcal{K}_{N_1\times \cdots \times N_d;M_1,...,M_d}| &, \mbox{ if } M_j \leq N_j \mbox{ for all } 1\leq j \leq d,\\
			0 &, \mbox{ otherwise.}
		\end{array}
		\right.
	\end{equation*}
	Then
	\begin{equation*}
		\begin{aligned}  
			&\lim_{N_1,...,N_d\to \infty}\frac{1}{N_1\cdots N_d}\sum_{M_1=1}^{N_1}\cdots \sum_{M_d=1}^{N_d} |\mathcal{K}_{N_1\times \cdots \times N_d;M_1,...,M_d} | \log a_{M_1,...,M_d}\\
			&=\lim_{N_1,...,N_d\to \infty}\frac{1}{N_1\cdots N_d}\sum_{M_1,...,M_d=1}^{\infty}\hat{K}_{N_1\times \cdots \times N_d;M_1,...,M_d}\log a_{M_1,...,M_d}.
		\end{aligned}
	\end{equation*}
	We claim that $\sum_{M_1,...,M_d=1}^{\infty} \frac{\hat{K}_{N_1\times \cdots \times N_d;M_1,...,M_d}\log a_{M_1,...,M_d}}{N_1\cdots N_d}$ converges uniformly in $N_1,...,N_d$ by Weierstrass M-test with
	\begin{equation*}
		\begin{aligned}
			\left|\frac{\hat{K}_{N_1\times \cdots \times N_d;M_1,...,M_d}\log a_{M_1,...,M_d}}{N_1\cdots N_d}  \right|&\leq \left| \frac{|\mathcal{J}_{N_1\times \cdots \times N_d;M_1,...,M_d}|\log a_{M_1,...,M_d}}{N_1\cdots N_d}  \right|\\
			&\leq \prod_{i=1}^d\left(\frac{1}{\ell^{(i)}_{M_i}}-\frac{1}{\ell^{(i)}_{M_i+1}}\right)\log a_{M_1,...,M_d}
		\end{aligned}
	\end{equation*}
	for all $N_1,...,N_d\geq 1$ and 
	\begin{equation*}
		\begin{aligned}
			&\sum_{M_1,...,M_d=1}^{\infty}\prod_{i=1}^d\left(\frac{1}{\ell^{(i)}_{M_i}}-\frac{1}{\ell^{(i)}_{M_i+1}}\right)\log a_{M_1,...,M_d} \\
			&\leq \sum_{M_1,...,M_d=1}^{\infty}\prod_{i=1}^d\left(\frac{1}{\ell^{(i)}_{M_i}}-\frac{1}{\ell^{(i)}_{M_i+1}}\right)\log C^{M_1\cdots M_d}<\infty, 
		\end{aligned}	
	\end{equation*}
	whenever $a_{M_1,...,M_d}\leq C^{M_1\cdots M_d}$ for all $M_1,...,M_d\geq 1$. 
	
	Thus,
	\begin{equation*}
		\begin{aligned}  
			&\lim_{N_1,...,N_d\to \infty}\frac{1}{N_1\cdots N_d}\sum_{M_1=1}^{N_1}\cdots \sum_{M_d=1}^{N_d} |\mathcal{K}_{N_1\times \cdots \times N_d;M_1,...,M_d} | \log a_{M_1,...,M_d}\\
			&=\lim_{N_1,...,N_d\to \infty}\frac{1}{N_1\cdots N_d}\sum_{M_1,...,M_d=1}^{\infty}\hat{K}_{N_1\times \cdots \times N_d;M_1,...,M_d}\log a_{M_1,...,M_d}\\
			&=\sum_{M_1,...,M_d=1}^{\infty}\lim_{N_1,...,N_d\to \infty}\frac{1}{N_1\cdots N_d}\hat{K}_{N_1\times \cdots \times N_d;M_1,...,M_d}\log a_{M_1,...,M_d}\\
			&=\sum_{M_1,...,M_d=1}^{\infty}\lim_{N_1,...,N_d\to \infty}\frac{1}{N_1\cdots N_d}| \mathcal{K}_{N_1\times \cdots \times N_d;M_1,...,M_d}|\log a_{M_1,...,M_d}.
		\end{aligned}
	\end{equation*}
	The proof is complete.
\end{proof}

Before state the main results, we need more definitions for semigroup ${\bf G}$ which is generated by vectors ${\bf p}_1,...,{\bf p}_k\in \mathbb{N}^d$. For $d,k\geq 1$, ${\bf p}_i=(p_{i1}, ..., p_{id}) \in \mathbb{N}^d, 1\leq i \leq k  $ and for each $1 \leq j \leq d$, $\gcd (p_{ij},p_{i'j})=1$ for all $1\leq i \neq i' \leq k$. Let ${\bf G}=\langle{\bf p}_1,..., {\bf p}_k\rangle=\{ {\bf 1}= l^{(j)}_1 \prec^{(j)} l^{(j)}_2 \prec^{(j)} \cdots \}$ be the semigroup generated by ${\bf p}_1, ..., {\bf p}_k$ with an order $\prec^{(j)}$ such that ${\bf G}$ follows $S^{(j)}$ for each $1 \leq j \leq d$. In particular, when $k=1$, ${\bf G}=\langle{\bf p}_1\rangle=\{ l_n={\bf p}_1^{n-1} \}_{n=1}^{\infty}$\footnote{${\bf p}^{n-1}_1=(p_{11}^{n-1},..., p_{1d}^{n-1})$}.

\section{The infinite-volume limit $\mu_{\infty}$ on $\mathbb{N}^d$ and Kolmogorov-Sinai entropy}
In this section, we prove the $\mathbb{N}^d$ version of Theorem 3.1 in \cite{chazottes2014thermodynamic} for layer stationarity and multiplication invariance (Theorems \ref{thm iid 2-multiple} and \ref{thm iid 2.5-multiple}), and then obtain the limit measures (Theorems \ref{thm mu 2-multiple} and \ref{thm mu 2.5-multiple}). We also give the $\mathbb{N}^d$ version of Lemma 3.4 in \cite{chazottes2014thermodynamic} (Lemmas \ref{lemma useful 1} and \ref{lemma useful 2}) to obtain the Kolmogorov-Sinai entropies (Theorems \ref{thm KS 2multi Nd} and \ref{thm Ks 2.5 Nd}). 
\subsection{Existence and invariance of the limit measure on $\mathbb{N}^d$}

\subsubsection{2-multiple Hamiltonian}
Let ${\bf p}_1=(p_1,...,p_d)\in \mathbb{N}^d$ and $\tau^{\bf i}=\{\tau_{\ell}^{\bf i}=\sigma_{{\bf p}^{\ell}_1\cdot {\bf i}}\}_{\ell=0}^{\infty}$ for each ${\bf i}\in \mathcal{I}_{{\bf p}_1}$. Then we have following result which is the $\mathbb{N}^d$ version of Theorem 3.1 in \cite{chazottes2014thermodynamic}.

\begin{theorem}\label{thm iid 2-multiple}
	Suppose that the $\{\tau^{\bf i} : {\bf i}\in \mathcal{I}_{{\bf p}_1} \}$ form an i.i.d. stationary processes, then the distribution of the corresponding $\sigma$ is multiplication invariant.
\end{theorem}

\begin{proof}
	We are going to show that for every finite collection of points ${\bf v}_1,...,{\bf v}_k \in \mathbb{N}^d$, and ${\bf m}\in \mathbb{N}^d$ the joint distribution of 
	\[
	(\sigma_{{\bf m}\cdot {\bf v}_1},...,\sigma_{{\bf m}\cdot {\bf v}_k})
	\]
	coincides with that of 
	\[
	(\sigma_{{\bf v}_1}, ..., \sigma_{{\bf v}_k}).
	\]
	First, we write ${\bf v}_j={\bf i}_j \cdot {\bf p}_1^{\ell_j}$ with ${\bf i}_j \in \mathcal{I}_{{\bf p}_1}$ for all $1\leq j \leq k$ and ${\bf m}={\bf i}\cdot {\bf p}_1^{\ell}$ with ${\bf i}\in \mathcal{I}_{{\bf p}_1}$. Denote ${\bf i}'_j\cdot {\bf p}_1^{\ell'_j}={\bf i} \cdot{\bf i}_j $ with ${\bf i}'_j\in \mathcal{I}_{{\bf p}_1}$ for all $1\leq j \leq k$.
	
	Then by the definition of $\tau$, we have to prove that the joint distribution of 
	\[
	\tau_{\ell_j+\ell+\ell'_j}^{{\bf i}'_j }, ~j\in\{1,...,k  \}
	\]
	coincides with that of 
	\[
	\tau_{\ell_j}^{{\bf i}_j}, ~ j\in\{1,...,k  \}.
	\]
	Denote $\{ {\bf i}_1, ... ,{\bf i}_k  \}=\{  {\bf i}_{n_1}, ...,  {\bf i}_{n_M} \}$ where $ {\bf i}_{n_i} \neq  {\bf i}_{n_j}$ for all $1\leq i \neq j \leq M$. For $1\leq w \leq M$, define
	\[
	X^w=\left( \tau_{\ell_j+\ell+\ell'_j}^{{\bf i}'_j} : 1\leq j \leq k, {\bf i}_j= {\bf i}_{n_w} \right)
	\]
	and 
	\[
	Y^w=\left( \tau_{\ell_j}^{{\bf i}_j} : 1\leq j \leq k, {\bf i}_j= {\bf i}_{n_w} \right).
	\]
	Then, by the independence of the different layers, the joint distribution of $\tau_{\ell_j+\ell+\ell'_j}^{{\bf i}'_j},  j\in\{1,...,k  \}$ and $\tau_{\ell_j}^{{\bf i}_j},   j\in\{1,...,k  \}$ coincide with the joint distribution of $\otimes_{w=1}^{M}X^w$ and $\otimes_{w=1}^{M} Y^w$ respectively, where $\otimes$ denotes independent joining.
	
	Therefore, it remains to show that for each $1\leq w \leq M$, the distributions of $X^w$ and $Y^w$ coincide. Since the layers $\tau^{{\bf i}'_j}$ and $\tau^{{\bf i}_j}$ are i.i.d., we have $X^w$ and $Y^w$ have the same distribution. The proof is complete. 
\end{proof}

\begin{remark}
	We remark that for ${\bf i},{\bf i}_1\neq{\bf i}_2\in \mathcal{I}_{{\bf p}_1}$ if ${\bf i}\cdot {\bf i}_1={\bf i}'_1\cdot {\bf p}_1^{\ell_1}$ and ${\bf i}\cdot {\bf i}_2={\bf i}'_2\cdot {\bf p}_1^{\ell_2}$ with ${\bf i}'_1,{\bf i}'_2 \in \mathcal{I}_{{\bf p}_1}$, then ${\bf i}'_1 \neq {\bf i}'_2$. When $\ell_1=\ell_2$, the statement easily true. If we assume $\ell_1>\ell_2$ and ${\bf i}'_1 = {\bf i}'_2$, then ${\bf i}\cdot {\bf i}_1={\bf i}\cdot {\bf i}_2\cdot {\bf p}_1^{\ell_1-\ell_2}$. That gives ${\bf i}_1= {\bf i}_2\cdot {\bf p}_1^{\ell_1-\ell_2}$ which contradicts with ${\bf i}_1 \in \mathcal{I}_{{\bf p}_1}$.
\end{remark}

As a consequence of Theorem \ref{thm iid 2-multiple} with the existence of the infinte-volume limit in each layer of $\tau$ spins we have the following results.

\begin{theorem}\label{thm mu 2-multiple}
	Let 
	\begin{equation*}
		\mu_{N_1\times \cdots \times N_d}(\sigma_{\mathcal{N}_{N_1\times \cdots \times N_d}})=\frac{e^{-H_{N_1\times \cdots \times N_d}(\sigma_{\mathcal{N}_{N_1\times \cdots \times N_d}})}}{\sum_{\sigma_{\ell_{\bf j(i)}}=\pm 1, {\bf i}\in \mathcal{N}_{N_1\times \cdots \times N_d}}e^{-H_{N_1\times \cdots \times N_d}(\sigma_{\mathcal{N}_{N_1\times \cdots \times N_d}})}} 
	\end{equation*}
	be the finite-volume probability measure corresponding to the Hamiltonian 
	\begin{equation*}
		H_{N_1\times \cdots \times N_d}(\sigma_{\mathcal{N}_{N_1\times \cdots \times N_d}})= -\beta\sum_{i_1=1}^{N_1} \cdots \sum_{i_d=1}^{N_d} \sigma_{l_{\bf j(i)}} \sigma_{l_{{\bf j(i)}+1}},
	\end{equation*}
where ${\bf j(i)}$ is the unique number such that ${\bf i}=(k_1, ...,k_d)\cdot l_{\bf j(i)}$ with $(k_1, ...,k_d)\in \mathcal{I}_{{\bf p}_1}$. Then	
	\item[1.]Unique limit measure : The measure $\mu_{N_1\times \cdots \times N_d}$ have a unique weak limit (as $N_1,...,N_d\rightarrow \infty$) denoted by $\mu_{\infty}$ which is Gibbs.
	\item[2.]Independent Ising layers : Under $\mu_{\infty}$, the $\tau^{\bf i}, {\bf i}\in \mathcal{I}_{{\bf p}_1}$, are independent and distributed according to the standard Ising model measure $\mu_\infty^{Ising}$ with a free boundary condition on the left.
	\item[3.]Multiplication invariance : The measure $\mu_{\infty}$ is multiplication invariant.
\end{theorem}
\subsubsection{Generalization of 2-multiple Hamiltonian to $S^{\bf G}_{N_1\times \cdots \times N_d}$ on $\mathbb{N}^d$}
Let $\tau^{\bf i}=\{\tau^{\bf i}_{k}=\sigma_{{\bf i}\cdot l^{(j)}_{k}}\}_{k=1}^{\infty}$ for each ${\bf i}\in \mathcal{I}_{{\bf p}_1,..., {\bf p}_k}$. For each $1\leq i \leq d$, $\gcd(p_{ji},p_{j'i})=1$ for all $1\leq j\neq j' \leq k$.

\begin{theorem}\label{thm iid 2.5-multiple}
	Suppose that the $\{\tau^{\bf i} : {\bf i}\in \mathcal{I}_{{\bf p}_1,..., {\bf p}_k} \}$ form an i.i.d. stationary processes, then the distribution of the corresponding $\sigma$ is multiplication invariant.
\end{theorem}

\begin{proof}	
	We have to show that for every finite collection of points ${\bf v}_1,...,{\bf v}_{N} \in \mathbb{N}^d$ and ${\bf m}\in \mathbb{N}^d$, the joint distribution of 
	\[
	(\sigma_{{\bf m}\cdot {\bf v}_1},...,\sigma_{{\bf m}\cdot {\bf v}_N})
	\]
	coincides with that of 
	\[
	(\sigma_{{\bf v}_1}, ..., \sigma_{{\bf v}_N}).
	\]
	First, we write ${\bf v}_j={\bf i}_j \cdot {\bf p}_1^{\ell_{j1}}\cdots {\bf p}_k^{\ell_{jk}}$ with ${\bf i}_j \in \mathcal{I}_{{\bf p}_1,..., {\bf p}_k}$ for all $1\leq j \leq N$ and ${\bf m}={\bf i}\cdot {\bf p}_1^{m_1} \cdots {\bf p}_k^{m_k}$ with ${\bf i}\in \mathcal{I}_{{\bf p}_1,..., {\bf p}_k}$. Note that ${\bf i}'_j\cdot{\bf p}_1^{\ell'_{j1}}\cdots {\bf p}_k^{\ell'_{jk}}={\bf i} \cdot{\bf i}_j$ with ${\bf i}'_j\in \mathcal{I}_{{\bf p}_1,..., {\bf p}_k}$ for all $1\leq j \leq k$. For convenience, we denote $\ell_j=(\ell_{j1},...,\ell_{jk})$, $\ell_j'=(\ell'_{j1},...,\ell'_{jk})$ for all $1\leq j \leq k$ and $m=(m_1,...,m_k)$.
	
	Then by the definition of $\tau$, we have to prove that the joint distribution of 
	\[
	\tau_{\ell_j+m+\ell'_j}^{{\bf i}'_j},~j\in\{1,...,k\}
	\]
	coincides with that of 
	\[
	\tau_{\ell_j}^{{\bf i}_j},~j\in\{1,...,k  \}.
	\]
	Denote $\{ {\bf i}_1, ... ,{\bf i}_k  \}=\{  {\bf i}_{n_1}, ...,  {\bf i}_{n_M} \}$ where $ {\bf i}_{n_i} \neq  {\bf i}_{n_j}$ for all $1\leq i \neq j \leq M$. For $1\leq w \leq M$, we define
	\[
	X^w=\left( \tau_{\ell_j+m+\ell'_j}^{{\bf i}'_j}: 1\leq j \leq k, {\bf i}_j= {\bf i}_{n_w} \right)
	\]
	and 
	\[
	Y^w=\left( \tau_{\ell_j}^{{\bf i}_j} : 1\leq j \leq k, {\bf i}_j= {\bf i}_{n_w} \right).
	\]
	Then, by the independence of the different layers, the joint distribution of $\tau_{\ell_j+m+\ell'_j}^{{\bf i}'_j},  j\in\{1,...,k  \}$ and $\tau_{\ell_j}^{{\bf i}_j},   j\in\{1,...,k  \}$ coincide with the joint distribution of $\otimes_{w=1}^{M}X^w$ and $\otimes_{w=1}^{M} Y^w$ respectively.
	
	Therefore, it remains to show that for each $1\leq w \leq M$, the distributions of $X^w$ and $Y^w$ coincide. Since the layers $\tau^{{\bf i}'_j}$ and $\tau^{{\bf i}_j}$ are i.i.d., we have that $X^w$ and $Y^w$ have the same distribution. The proof is complete. 
\end{proof}

\begin{remark}
	If ${\bf i}, {\bf i}_1\neq {\bf i}_2 \in  \mathcal{I}_{{\bf p}_1,..., {\bf p}_k}$ and ${\bf i}\cdot {\bf i}_1={\bf i}'_1 \cdot {\bf p}_1^{\ell'_{11}}\cdots {\bf p}_k^{\ell'_{1k}}, {\bf i}\cdot {\bf i}_2={\bf i}'_2 \cdot {\bf p}_1^{\ell'_{21}}\cdots {\bf p}_k^{\ell'_{2k}}$ with ${\bf i}'_1,{\bf i}'_2 \in  \mathcal{I}_{{\bf p}_1,..., {\bf p}_k}$, then ${\bf i}'_1 \neq {\bf i}'_2$. When $(\ell'_{11},...,\ell'_{1k})=(\ell'_{21},...,\ell'_{2k})$, the statement clearly true. If $(\ell'_{11},...,\ell'_{1k})\neq(\ell'_{21},...,\ell'_{2k})$ and ${\bf i}'_1 = {\bf i}'_2$, without loss of generality, we may assume $\ell'_{11}>\ell'_{21}$. This implies ${\bf i}\cdot {\bf i}_1\cdot {\bf c}_1={\bf i}\cdot {\bf i}_2\cdot{\bf p}_1^{\ell'_{11}-\ell'_{21}} \cdot{\bf c}_2$, then $ {\bf i}_1\cdot{\bf c}_1= {\bf i}_2\cdot{\bf p}_1^{\ell'_{11}-\ell'_{21}}\cdot{\bf c}_2$ where ${\bf c}_1, {\bf c}_2$ are the vectors only multipliable by ${\bf p}_i$'s for $2\leq i \leq k$. Then the coprime property of ${\bf p}_1$ and ${\bf p}_i, 2\leq  i \leq k$ gives ${\bf p}_1 \mid {\bf i}_1$ which contradicts ${\bf i}_1 \in  \mathcal{I}_{{\bf p}_1,..., {\bf p}_k}$.
	
\end{remark}

As a consequence of Theorem \ref{thm iid 2-multiple} with the existence of the infinte-volume limit in each layer of $\tau$ spins we have the following theorem.

\begin{theorem}\label{thm mu 2.5-multiple}
	Let 
	\begin{equation*}
		\mu_{N_1\times \cdots \times N_d}(\sigma_{\mathcal{N}_{N_1\times \cdots \times N_d}})=\frac{e^{-H_{N_1\times \cdots \times N_d}(\sigma_{\mathcal{N}_{N_1\times \cdots \times N_d}})}}{\sum_{\sigma_{l^{(j)}_{\bf j(i)}}=\pm 1, {\bf i}\in \mathcal{N}_{N_1\times \cdots \times N_d}}e^{-H_{N_1\times \cdots \times N_d}(\sigma_{\mathcal{N}_{N_1\times \cdots \times N_d}})}} 
	\end{equation*}
	be the finite-volume probability measure corresponding to the Hamiltonian 
	\begin{equation*}
		H_{N_1\times \cdots \times N_d}(\sigma_{\mathcal{N}_{N_1\times \cdots \times N_d}})= -\beta\sum_{i_1=1}^{N_1} \cdots \sum_{i_d=1}^{N_d} \sigma_{l^{(j)}_{\bf j(i)}} \sigma_{l^{(j)}_{{\bf j(i)}+1}},
	\end{equation*}
 where ${\bf j(i)}$ is the unique number such that ${\bf i}=(k_1, ...,k_d)\cdot l^{(j)}_{\bf j(i)}$ with $(k_1, ...,k_d)\in \mathcal{I}_{{\bf p}_1,..., {\bf p}_k}$. Then	
	\item[1.]Unique limit measure : The measure $\mu_{N_1\times \cdots \times N_d}$ have a unique weak limit with respect to order $\prec^{(j)}$ denoted by $\mu^{(j)}_{\infty}$ which is Gibbs.
	\item[2.]Independent Ising layers : Under $\mu^{(j)}_{\infty}$, the $\tau^{\bf i}, {\bf i}\in \mathcal{I}_{{\bf p}_1,...,{\bf p}_k}$ are independent and distributed according to the standard Ising model measure $\mu_\infty^{Ising}$ with a free boundary condition on the left.
	\item[3.]Multiplication invariance : The measure $\mu^{(j)}_{\infty}$ is multiplication invariant.
\end{theorem}

\begin{remark}
	We remark that the order $\prec^{(j)}$ leads $\mu_{N_1\times \cdots \times N_d}$ to be Markov on each independent layer $\tau^{\bf i}, {\bf i}\in \mathcal{I}_{{\bf p}_1,...,{\bf p}_k} $.
\end{remark}

\subsection{Kolmogorov-Sinai entropy with respect to the limit measure on $\mathbb{N}^d$}

\subsubsection{2-multiple Hamiltonian}
Recall that $\mathcal{L}_{N_1\times \cdots \times N_d  }({\bf i})=\mathcal{M}_{{\bf p}_1}({\bf i})\cap \mathcal{N}_{N_1\times \cdots \times N_d }$ is the subset of $\mathcal{M}_{{\bf p}_1}({\bf i})$, which belongs to the $\mathcal{N}_{N_1\times \cdots \times N_d }$ lattice. Then we have results as follows.
\begin{lemma}\label{lemma useful 1}
	Let $\phi : \mathbb{N}\rightarrow\mathbb{R}$ be a measurable function such that there exist $C'>0$ and $r>0$ such that $|\phi(n)|\leq C'n^r$ for all $n\in \mathbb{N}$. Then we have
	\begin{equation*}
		\begin{aligned}
			\lim_{N_1,...,N_d\rightarrow\infty}\frac{1}{N_1\cdots N_d}\sum_{{\bf i}\in \mathcal{I}_{\bf p}}\phi(|\mathcal{L}_{N_1\times \cdots \times N_d  }({\bf i})|)=\sum_{\ell=1}^{\infty}\frac{(p_1\cdots p_d -1)^2}{(p_1\cdots p_d)^{\ell+1}}\phi(\ell).
		\end{aligned}
	\end{equation*} 
\end{lemma}

\begin{proof}
	The proof is a direct consequence of Lemma \ref{lemma Nd decomposition 2.5 multiple}. More precisely,
	\begin{equation*}
		\begin{aligned}
			&	\lim_{N_1,...,N_d\rightarrow\infty}\frac{1}{N_1\cdots N_d}\sum_{{\bf i}\in \mathcal{I}_{\bf p}}\phi(|\mathcal{L}_{N_1\times \cdots \times N_d  }({\bf i})|)\\
			&=\lim_{N_1,...,N_d\rightarrow\infty}\frac{1}{N_1\cdots N_d}\sum_{\ell=1}^{K(N_1,...,N_d)}\left( 1-\frac{1}{p_1\cdots p_d}\right)\left( \frac{N_1\cdots N_d}{(p_1\cdots p_d)^{\ell-1}}-\frac{N_1\cdots N_d}{(p_1\cdots p_d)^{\ell}}\right)\phi(\ell)\\	
			&=\lim_{N_1,...,N_d\rightarrow\infty}\frac{1}{N_1\cdots N_d}\sum_{\ell=1}^{K(N_1,...,N_d)}\frac{N_1\cdots N_d(p_1\cdots p_d-1)^2}{(p_1\cdots p_d)^{\ell+1}}\phi(\ell)\\
			&=\sum_{\ell=1}^{\infty}\frac{(p_1\cdots p_d-1)^2}{(p_1\cdots p_d)^{\ell+1}}\phi(\ell),
		\end{aligned}
	\end{equation*} 
	where the last equality holds by Weierstrass M-test with $|\phi(\ell)|\leq C'\ell^r$ and $K(N_1,...,N_d)$ is the maximum cardinality of $\mathcal{L}_{N_1\times \cdots \times N_d  }({\bf i})$ for all ${\bf i}\in \mathcal{I}_{\bf p}$.
\end{proof}

Theorem \ref{thm KS 2multi Nd} below is an immediate consequence of Lemma \ref{lemma useful 1}.

\begin{theorem}\label{thm KS 2multi Nd}
	The explicit formula for the KS entropy of $\mu_\infty$ is
	\begin{equation*}
		\begin{aligned}
		&-\lim_{N_1,...,N_d\rightarrow\infty}\frac{1}{N_1\cdots N_d} \mathbb{E}_{\mu_\infty} \log \mu_{\infty}(\sigma_{\mathcal{N}_{N_1\times \cdots \times N_d}})\\
		&=-\sum_{k=1}^{\infty} \frac{(p_1\cdots p_d-1)^2}{(p_1\cdots p_d)^{k+1}} \mathbb{E}_{\mu_\infty^{Ising}} \log \mu_\infty^{Ising} (\tau_0,...,\tau_{k-1}).
\end{aligned}	
\end{equation*}
\end{theorem}
\subsubsection{Generalization of 2-multiple Hamiltonian to $S^{\bf G}_{N_1\times \cdots \times N_d}$ on $\mathbb{N}^d$}
For each $1\leq j \leq d$ and ${\bf i}\in  \mathcal{I}_{{\bf p}_1,...,{\bf p}_k}\cap \mathcal{N}_{N_1\times\cdots \times N_d}$, let $\mathcal{L}^{(j)}_{N_1\times \cdots \times N_d  }({\bf i})$ be the subset of $\mathcal{M}_{{\bf p}_1,...,{\bf p}_k}({\bf i})$ which satisfies $i_j p_{1j}^{\ell_1}\cdots p_{kj}^{\ell_k} \leq N_j$. Then the similar result of Lemma \ref{lemma useful 1} is obtained.

\begin{lemma}\label{lemma useful 2}
	Let $\phi : \mathbb{N}\rightarrow\mathbb{R}$ be a measurable function such that there exist $C'>0$ and $r>0$ such that $|\phi(n)|\leq C'n^r$ for all $n\in \mathbb{N}$. Then we have
	\begin{equation*}
		\begin{aligned}
			&\lim_{N_1,...,N_d\rightarrow\infty}\frac{1}{N_1\cdots N_d}\sum_{{\bf i}\in \mathcal{I}_{{\bf p}_1,...,{\bf p}_k}}\phi(|\mathcal{L}^{(j)}_{N_1\times \cdots \times N_d  }({\bf i})|)\\
			&=\prod_{i=1}^k \left(1-\frac{1}{p_{i1}\cdots p_{id}}\right)\sum_{M_1,...,M_d=1}^{\infty}\left[ \prod_{i=1}^{d}\left(\frac{1}{\ell_{M_i}^{(i)}}-\frac{1}{\ell_{M_i+1}^{(i)}}\right)\right]\phi(M_j).
		\end{aligned}
	\end{equation*} 
\end{lemma}

\begin{proof}
	Applying the same argument of the Lemma \ref{lemma useful 1} with Lemma \ref{lemma Nd decomposition 2.5 multiple}, the proof is complete.
\end{proof}

By Lemma \ref{lemma useful 2}, we have the following result.

\begin{theorem}\label{thm Ks 2.5 Nd}
	For each $1\leq j \leq d$, the explicit formula for the KS entropy of $\mu^{(j)}_\infty$ is
	\begin{equation*}
		\begin{aligned}
			&-\lim_{N_1,...,N_d\rightarrow \infty} \frac{1}{N_1\cdots N_d} \mathbb{E}_{\mu_\infty} \log \mu^{(j)}_{\infty}(\sigma_{\mathcal{N}_{N_1\times \cdots \times N_d}})\\
			&=-C\sum_{k=1}^{\infty} \left( \frac{1}{\ell^{(j)}_k} -\frac{1}{\ell^{(j)}_{k+1}}  \right) \mathbb{E}_{\mu_\infty^{Ising}} \log \mu_\infty^{Ising} (\tau_0,...,\tau_{k-1}),
		\end{aligned}
	\end{equation*}
where $C$ is defined in Theorem \ref{theorem main Nd 2.5multiple}.
\end{theorem}

\section{Free energy functions and large deviation principle}
In this section, we obtain the generalization of Theorem \ref{thm 1.2} (Theorem \ref{thm main 1d 2.5multiple}), and we consider two types of generalizations of Theorem \ref{thm 1.3} (Theorems \ref{thm general} and \ref{theorem main Nd 2.5multiple}).

\subsection{Generalization of 2-multiple Hamiltonian to $S^G_N$ on $\mathbb{N}$}
Let $k\geq 1$, $p_1,p_2,...,p_k$ be co-primes, and $G=\langle p_1,p_2,...,p_k \rangle=\{1=\ell_1<\ell_2< \cdots   \}$ be a semigroup generated by $p_1,p_2,...,p_k$. Denote $\gamma(G)=\sum_{i=1}^\infty \frac{1}{\ell_i}$. The LDP of $S^G_N=\sum_{i=1}^N \sigma_{\ell_{j(i)}} \sigma_{\ell_{j(i)+1}}$, $j(i)$ is the unique number such that $i=i' \ell_j$ and $p_n \nmid i'$ for all $1\leq n \leq k$, is presented in the following theorem.

\begin{theorem}\label{thm main 1d 2.5multiple}
	The following statements hold true.
	\item[1.]The explicit expression of the free energy function associated to the multiple sum $S^G_N$ is
	\begin{equation*}
		\begin{aligned}
			F_r(\beta)=\frac{1+\gamma(G)^{-1}}{2}\log (r(1-r))+\gamma(G)^{-1}\log|v^T\cdot e_+|^2+\log\Lambda_++\mathcal{G}(\beta),
		\end{aligned}
	\end{equation*}
	where 
	\begin{equation*}
		\begin{aligned}
			\mathcal{G}(\beta)=\gamma(G)^{-1}\sum_{k=1}^{\infty}\left( \frac{1}{\ell_k} -\frac{1}{\ell_{k+1}} \right)\log\left(1+\left( \frac{2\cosh(h)}{|v^T\cdot e_+|^2}-1 \right)\left(\frac{\Lambda_-}{\Lambda_+} \right)^k\right).	
		\end{aligned}
	\end{equation*}
	\item[2.]The function $F_r(\beta)$ is differentiable with respect to $\beta \in \mathbb{R}$.
	\item[3.]The multiple average $\frac{S^G_N}{N}$ satisfies a LDP with rate function given by
	\begin{equation*}
		I_{r}(x)=\sup_{\beta \in \mathbb{R}}\left( \beta x-F_r(\beta) \right).
	\end{equation*}
	Furthermore, if $(F_r)'(\eta)=y$, then $I_{r}(y)=\eta y- F_r(\eta)$. 
\end{theorem}

\begin{proof}
	\item[1.]By the $\mathbb{N}^d$ version of Lemma 2.7 in \cite{ban2021entropy} and the similar argument of Theorem 3.2 in \cite{ban2021LDP}, we obtain
	\begin{equation*}
		\begin{aligned}
			F_r(\beta)&=B\sum_{k=1}^{\infty} \left( \frac{1}{\ell_k} -\frac{1}{\ell_{k+1}} \right) \log (r(1-r))^{\frac{k+1}{2}}Z(\beta,h,k+1)\\
			&=B\sum_{k=1}^{\infty} \left( \frac{1}{\ell_k} -\frac{1}{\ell_{k+1}} \right) \log (r(1-r))^{\frac{k+1}{2}}\\
			&+B\sum_{k=1}^{\infty} \left( \frac{1}{\ell_k} -\frac{1}{\ell_{k+1}} \right)\log\left(|v^T\cdot e_+|^2 \Lambda_+^k+(2\cosh(h)-|v^T\cdot e_+|^2) \Lambda_-^k\right)\\
			&=B\frac{1+\gamma(G)}{2}\log (r(1-r))+B\log|v^T\cdot e_+|^2+B\gamma(G)\log\Lambda_++\mathcal{G}(\beta)\\
			&=\frac{1+\gamma(G)^{-1}}{2}\log (r(1-r))+\gamma(G)^{-1}\log|v^T\cdot e_+|^2+\log\Lambda_++\mathcal{G}(\beta),
		\end{aligned}
	\end{equation*}
	where the constant $B=\prod_{i=1}^k \left(1-\frac{1}{p_i}  \right)=\gamma(G)^{-1}$,
	 \begin{equation*}
	 Z(\beta,h,k+1)=v^T \left[\begin{matrix}
	 	e^{\beta+h}&e^{-\beta}\\
	 	e^{-\beta}&e^{\beta-h}
	 \end{matrix} \right]^k v
	 \end{equation*}
	 and
	\begin{equation*}
		\begin{aligned}
			\mathcal{G}(\beta)=\gamma(G)^{-1}\sum_{k=1}^{\infty}\left( \frac{1}{\ell_k} -\frac{1}{\ell_{k+1}} \right)\log\left(1+\left( \frac{2\cosh(h)}{|v^T\cdot e_+|^2}-1 \right)\left(\frac{\Lambda_-}{\Lambda_+} \right)^k\right).	
		\end{aligned}
	\end{equation*}
	\item[2.]The proof is similar to the (2) of Theorem 3.2 in \cite{ban2021LDP}. More pricisely, we are going to show the sum
	\begin{equation*}
		\sum_{k=1}^{\infty}\left( \frac{1}{\ell_k} -\frac{1}{\ell_{k+1}} \right)\left[\log\left(1+\left( \frac{2\cosh(h)}{|v^T\cdot e_+|^2}-1 \right)\left(\frac{\Lambda_-}{\Lambda_+} \right)^k\right)\right]'
	\end{equation*}
	coverge uniformly with respect to $\beta \in \mathbb{R}$, where the notation $'$ stays for the derivative with respect to $\beta$. Then, we apply the same reasoning on
	\begin{equation*}
		\left[\log\left(1+\left( \frac{2\cosh(h)}{|v^T\cdot e_+|^2}-1 \right)\left(\frac{\Lambda_-}{\Lambda_+} \right)^k\right)\right]'
	\end{equation*}
	and then apply the Weierstrass M-test, and the proof is complete.
	\item[3.]The proof is a direct consequence of Theorem \ref{thm main 1d 2.5multiple} (2) and G\"{a}rtner-Ellis Theorem \cite{ban2021LDP}. 
\end{proof}

\begin{corollary}
	\item[1.]For $k=1$, $p_1 \geq 1$, $G=\langle p_1 \rangle=\{ p_1^{i-1} : i\in \mathbb{N}  \}$ and $\gamma(S)=\frac{p_1}{p_1-1}$. The explicit expression of the free energy function associated to the multiple sum $S^G_N$ is
	\begin{equation*}
		\begin{aligned}
			F_r(\beta)=\frac{2p_1-1}{2p_1}\log (r(1-r))+\frac{2(p_1-1)}{p_1} \log|v^T\cdot e_+|+\log\Lambda_++\mathcal{G}(\beta),
		\end{aligned}
	\end{equation*}
	where 
	\begin{equation*}
		\begin{aligned}
			\mathcal{G}(\beta)=\frac{p_1-1}{p_1}\sum_{k=1}^{\infty}\frac{p_1-1}{p_1^k}  \log\left(1+\left( \frac{2\cosh(h)}{|v^T\cdot e_+|^2}-1 \right)\left(\frac{\Lambda_-}{\Lambda_+} \right)^k\right).	
		\end{aligned}
	\end{equation*}		
	\item[2.]In addition, $p_1=2$, $G=\langle 2 \rangle=\{ \ell_i=2^{i-1}: i\in \mathbb{N}  \}$ and $\gamma(S)=2$. The explicit expression of the free energy function associated to the multiple sum $S^G_N$ is
	\begin{equation*}
		\begin{aligned}
			F_r(\beta)=\frac{3}{4}\log (r(1-r))+\log|v^T\cdot e_+|+\log\Lambda_++\mathcal{G}(\beta),
		\end{aligned}
	\end{equation*}
	where 
	\begin{equation*}
		\begin{aligned}
			\mathcal{G}(\beta)=\frac{1}{2}\sum_{k=1}^{\infty}\frac{1}{2^k}  \log\left(1+\left( \frac{2\cosh(h)}{|v^T\cdot e_+|^2}-1 \right)\left(\frac{\Lambda_-}{\Lambda_+} \right)^k\right).	
		\end{aligned}
	\end{equation*}
	Which is coincides with the resilt in Theorem \ref{thm 1.2}.
\end{corollary}

\begin{example}
		\begin{figure}
	\includegraphics[scale=0.3]{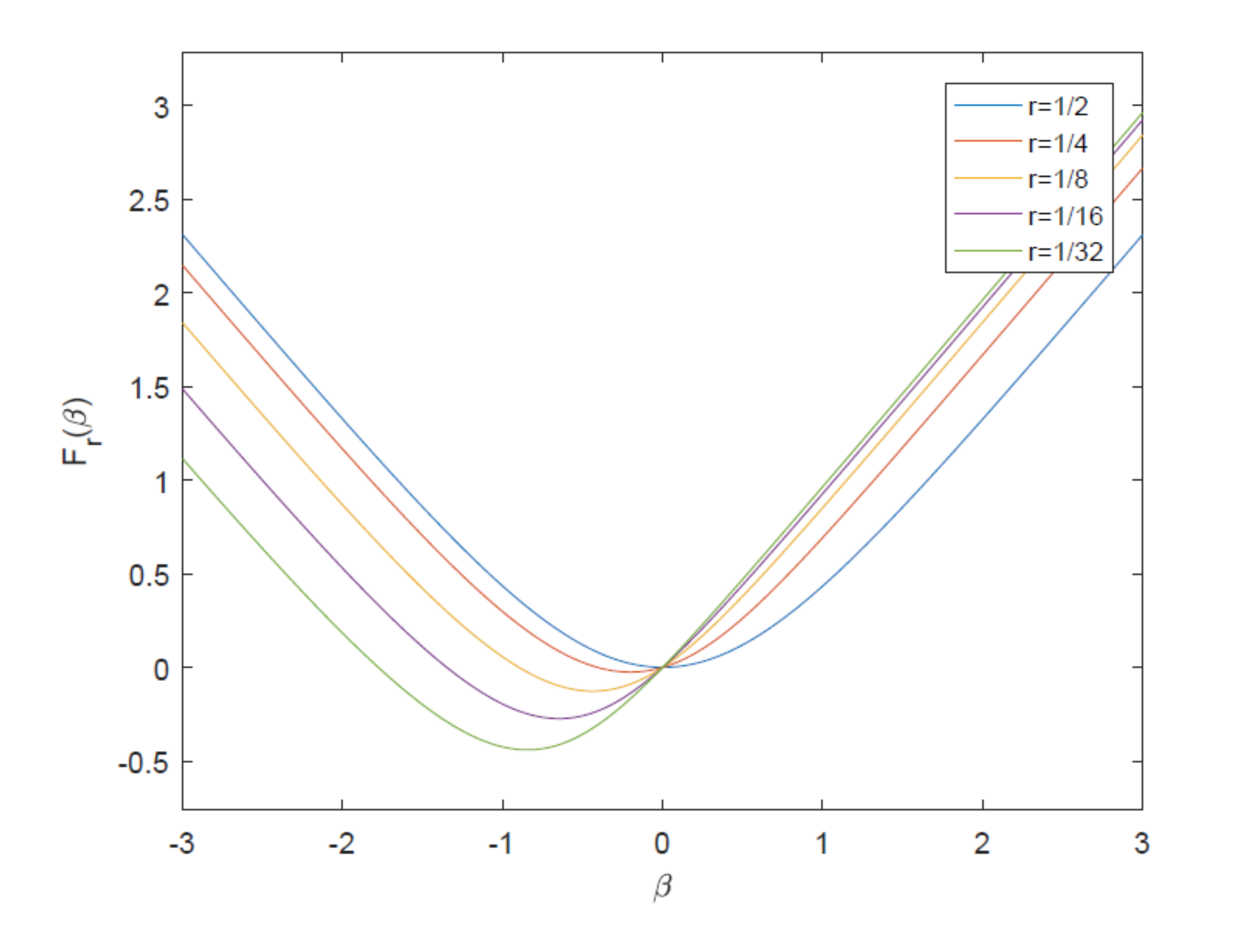}
	\caption{Plot of the $F_r(\beta)$ with $p_i, 1\leq i \leq 5$ for different $r$ values.}
	\label{fig 1}
\end{figure}	
	Figure \ref{fig 1} illustrates the free energy function for different $r\in (0,1)$ which is obtained from Theorem \ref{thm main 1d 2.5multiple} by truncating the sum to the first 100 terms. 
	The graph is obtained for the case $d = 1$ with $p_1=2,p_2=3,p_3=5,p_4=7,p_5=11$, $G=\langle 2,3,5,7,11 \rangle$ and $\gamma(G)^{-1}=\left( 1-\frac{1}{2} \right) \left( 1-\frac{1}{3} \right)\left( 1-\frac{1}{5} \right)\left( 1-\frac{1}{7} \right)\left( 1-\frac{1}{11} \right)=\frac{16}{77}$.
\end{example}

\subsection{Generalization of 2-multiple Hamiltonian to $S^{\bf G}_{N_1\times \cdots \times N_d}$ on $\mathbb{N}^d$}\label{section 4.2}
 Recall that $S^{(j)}=\langle p_{1j}, ..., p_{kj} \rangle=\{ 1=\ell^{(j)}_1 < \ell^{(j)}_2 <\cdots   \}$ is the semigroup generated by $p_{1j}, ..., p_{kj}$. And ${\bf G}=\langle{\bf p}_1,..., {\bf p}_k\rangle=\{ {\bf 1}= l^{(j)}_1 \prec^{(j)} l^{(j)}_2 \prec^{(j)} \cdots \}$ be the semigroup generated by ${\bf p}_1, ..., {\bf p}_k$ with an order $\prec^{(j)}$ such that ${\bf G}$ follows $S^{(j)}$ for each $1 \leq j \leq d$. The LDP of $S^{\bf G}_{N_1 \times \cdots \times N_d}=\sum_{i_i=1}^{N_1} \cdots \sum_{i_d=1}^{N_d} \sigma_{l^{(j)}_{\bf j(i)}} \sigma_{l^{(j)}_{{\bf j(i)}+1}}$, ${\bf j(i)}$ is the unique number such that ${\bf i}=(k_1, ...,k_d)\cdot l^{(j)}_{\bf j(i)}$ with $(k_1, ...,k_d)\in \mathcal{I}_{{\bf p}_1,..., {\bf p}_k}$, is described by the following energy function. 

The \emph{directional free energy function} is defined as following. For $1\leq j \leq d$, 
\[
F_r(\beta):=\lim_{N_1,...,N_d\rightarrow \infty}\log \mathbb{E}_r \left( e^{ \beta\sum_{i_1=1}^{N_1} \cdots \sum_{i_d=1}^{N_d} \sigma_{l^{(j)}_{\bf j(i)}} \sigma_{l^{(j)}_{{\bf j(i)}+1}}}\right).
\]
The following result takes the sum over all elements ${\bf i}$ if it is less than the maximum element on its layer intersection with the $N_1 \times \cdots \times N_d $ lattice (in the $\prec^{(j)}$ sense). That is, we take the sum over all ${\bf i}\in  \mathcal{M}_{{\bf p}_1,...,{\bf p}_k}({\bf j(i)}) $ if 
\begin{equation*}
{\bf i}  \prec^{(j)} \max \left\{ {\bf i}' : {\bf i}'\in  \mathcal{M}_{{\bf p}_1,...,{\bf p}_k}({\bf j(i)}) \cap \mathcal{N}_{N_1 \times \cdots \times N_d}   \right\}.
\end{equation*}

\begin{theorem}\label{thm general}
	For $1\leq j \leq d$, the explicity formula of the directional energy function associated to the multiple sum $S^{\bf G}_{N_1\times\cdots \times N_d}$ converges to 
	\begin{equation*}
	F_r(\beta)=\prod_{i=1}^{k} \left(1-\frac{1}{p_{i1}\cdots p_{id}}\right)\sum_{k_1,...,k_d=1}^{\infty} \prod_{i=1}^{d}\left( \frac{1}{\ell^{(i)}_{k_i}} -\frac{1}{\ell^{(i)}_{k_i+1}} \right)\log (r(1-r))^{\frac{b_{k_1,...,k_d}+1}{2}}Z(\beta,h,b_{k_1,...,k_d}+1),
	\end{equation*}
	where $b_{k_1,...,k_d}$ is the number of the elements in ${\bf G}$ less than or equal to the maximum element (with the order $\prec^{(j)}$) in the $\ell^{(1)}_{k_1}\times \cdots \times \ell^{(d)}_{k_d}$ lattice.
\end{theorem}

\begin{proof}
	The proof is directly by the Lemmas \ref{lemma Nd 2.4} and \ref{lemma Nd decomposition 2.5 multiple} with the observation $\left| b_{k_1,...,k_d}  \right| \leq k_j$.
\end{proof}

Due to the explicity expression of Theorem \ref{thm general} is difficult to obtain, we consider the following type directional free energy function which takes the sum over all ${\bf i}\in  \mathcal{I}_{{\bf p}_1,...,{\bf p}_k}\cap \mathcal{N}_{N_1\times\cdots \times N_d}$ layers with the layer members ${\bf i}\cdot {\bf p}_1^{\ell_1} \cdots {\bf p}_k^{\ell_k}$ satisfy $i_j p_{1j}^{\ell_1}\cdots p_{kj}^{\ell_k} \leq N_j$. Then we have following result.

\begin{theorem}\label{theorem main Nd 2.5multiple}
	For $1\leq j \leq d$,
	\item[1.]The explicity expression of the directional energy function associated to the multiple sum $S^{\bf G}_{N_1\times\cdots \times N_d}$ is
	\begin{equation*}
		\begin{aligned}
			F_r(\beta)= \frac{C+C\gamma(S^{(j)})}{2}\log (r(1-r))+C\log|v^T\cdot e_+|^2 +C\gamma(S^{(j)})\log\Lambda_++\mathcal{G}(\beta),
		\end{aligned}
	\end{equation*}
	where $C=\prod_{i=1}^{k} \left(1-\frac{1}{p_{i1}\cdots p_{id}}\right)  \prod_{1\leq i\neq j \leq d}\gamma(S^{(i)}) $ and
	\begin{equation*}
		\mathcal{G}(\beta)=C\sum_{k=1}^{\infty}\left( \frac{1}{\ell^{(j)}_k} -\frac{1}{\ell^{(j)}_{k+1}} \right)\log\left(1+\left( \frac{2\cosh(h)}{|v^T\cdot e_+|^2}-1 \right)\left(\frac{\Lambda_-}{\Lambda_+} \right)^k\right).
	\end{equation*}
	\item[2.]The function $F_r(\beta)$ is differentiable with respect to $\beta \in \mathbb{R}$.
	\item[3.]The multiple average $\frac{S^{\bf G}_{N_1 \times \cdots \times N_d}}{N_1 \cdots N_d}$ satisfies a LDP with rate function given by
	\begin{equation*}
		I_{r}(x)=\sup_{\beta \in \mathbb{R}}\left( \beta x-F_r(\beta) \right).
	\end{equation*}
	Furthermore, if $(F_r)'(\eta)=y$, then $I_{r}(y)=\eta y- F_r(\eta)$. 
\end{theorem}

\begin{proof}
	\item[1.]By Lemmas \ref{lemma Nd 2.4} and Lemma \ref{lemma Nd decomposition 2.5 multiple} with the similar process of Theorem 3.2 in \cite{ban2021LDP}, we have 
	\begin{equation*}
		\begin{aligned}
			F_r(\beta)
			&=\lim_{N_1,...,N_d\rightarrow \infty}\log \mathbb{E}_r \left( e^{ \beta\sum_{i_1=1}^{N_1} \cdots \sum_{i_d=1}^{N_d} \sigma_{l^{(j)}_{\bf j(i)}} \sigma_{l^{(j)}_{{\bf j(i)}+1}}}\right)\\
			&=\prod_{i=1}^{k} \left(1-\frac{1}{p_{i1}\cdots p_{id}}\right)\sum_{k_1=1}^{\infty}\cdots \sum_{k_d=1}^{\infty} \prod_{i=1}^{d}\left( \frac{1}{\ell^{(i)}_{k_i}} -\frac{1}{\ell^{(i)}_{k_i+1}} \right)\log (r(1-r))^{\frac{k_j+1}{2}}Z(\beta,h,k_j+1).
		\end{aligned}
	\end{equation*}
	Then replace $k_j$ by $k$, we obtain			
	\begin{equation*}
		\begin{aligned}			
			F_r(\beta)&=C\sum_{k=1}^{\infty} \left( \frac{1}{\ell^{(j)}_k} -\frac{1}{\ell^{(j)}_{k+1}} \right)  \log (r(1-r))^{\frac{k+1}{2}}Z(\beta,h,k+1)\\
			&= \frac{C+C\gamma(S^{(j)})}{2}\log (r(1-r))+C\log|v^T\cdot e_+|^2 +C\gamma(S^{(j)})\log\Lambda_++\mathcal{G}(\beta),
		\end{aligned}
	\end{equation*}
	where $C=\prod_{i=1}^{k} \left(1-\frac{1}{p_{i1}\cdots p_{id}}\right)  \prod_{1\leq i\neq j \leq d}\gamma(S^{(i)}) $ and
	\begin{equation*}
		\mathcal{G}(\beta)=C\sum_{k=1}^{\infty}\left( \frac{1}{\ell^{(j)}_k} -\frac{1}{\ell^{(j)}_{k+1}} \right)\log\left(1+\left( \frac{2\cosh(h)}{|v^T\cdot e_+|^2}-1 \right)\left(\frac{\Lambda_-}{\Lambda_+} \right)^k\right).
	\end{equation*}
	\item[2.]By a similar argument of Theorem \ref{thm main 1d 2.5multiple} (2), we have $F_r(\beta)$ is differentiable with respect to $\beta \in \mathbb{R}$.
	\item[3.]The proof is a direct consequence of Theorem \ref{theorem main Nd 2.5multiple} (2) and G\"{a}rtner-Ellis Theorem \cite{ban2021LDP}. 
\end{proof}

\begin{corollary}
For $k=1, {\bf p}_1=(p_1\cdots p_d,1,...,1)\in \mathbb{N}^d$ and $j=1$, then $C= \left(1-\frac{1}{p_{1}\cdots p_{d}}\right)$, $S^{(1)}=\langle p_1\cdots p_d\rangle=\{ \ell^{(1)}_n=(p_1\cdots p_d)^{n-1} \}_{n=1}^{\infty}$ and $\gamma(S^{(1)})=\frac{p_1\cdots p_d}{p_1\cdots p_d-1}$. These give the directional energy function associated to the sum $S^{{\bf p}_1}_{N_1\times\cdots \times N_d}$ is
		\begin{equation*}
	\begin{aligned}
	F_r(\beta)
	=\frac{2p_1\cdots p_d-1}{2p_1\cdots p_d}\log (r(1-r))+\frac{p_1\cdots p_d-1}{p_1\cdots p_d}\log|v^T\cdot e_+|^2 +\log\Lambda_++\mathcal{G}(\beta),
	\end{aligned}
	\end{equation*}
where
\begin{equation*}
\begin{aligned}
\mathcal{G}(\beta)=\sum_{\ell=1}^{\infty}\frac{(p_1\cdots p_d-1)^2}{(p_1\cdots p_d)^{\ell+1}}\log\left(1+\left( \frac{2\cosh(h)}{|v^T\cdot e_+|^2}-1 \right)\left(\frac{\Lambda_-}{\Lambda_+} \right)^\ell\right).
\end{aligned}
\end{equation*}	
Which is coincides with the formula in Theorem \ref{thm 1.3}.
\end{corollary}

\begin{example}
		\begin{figure}
		\includegraphics[scale=0.3]{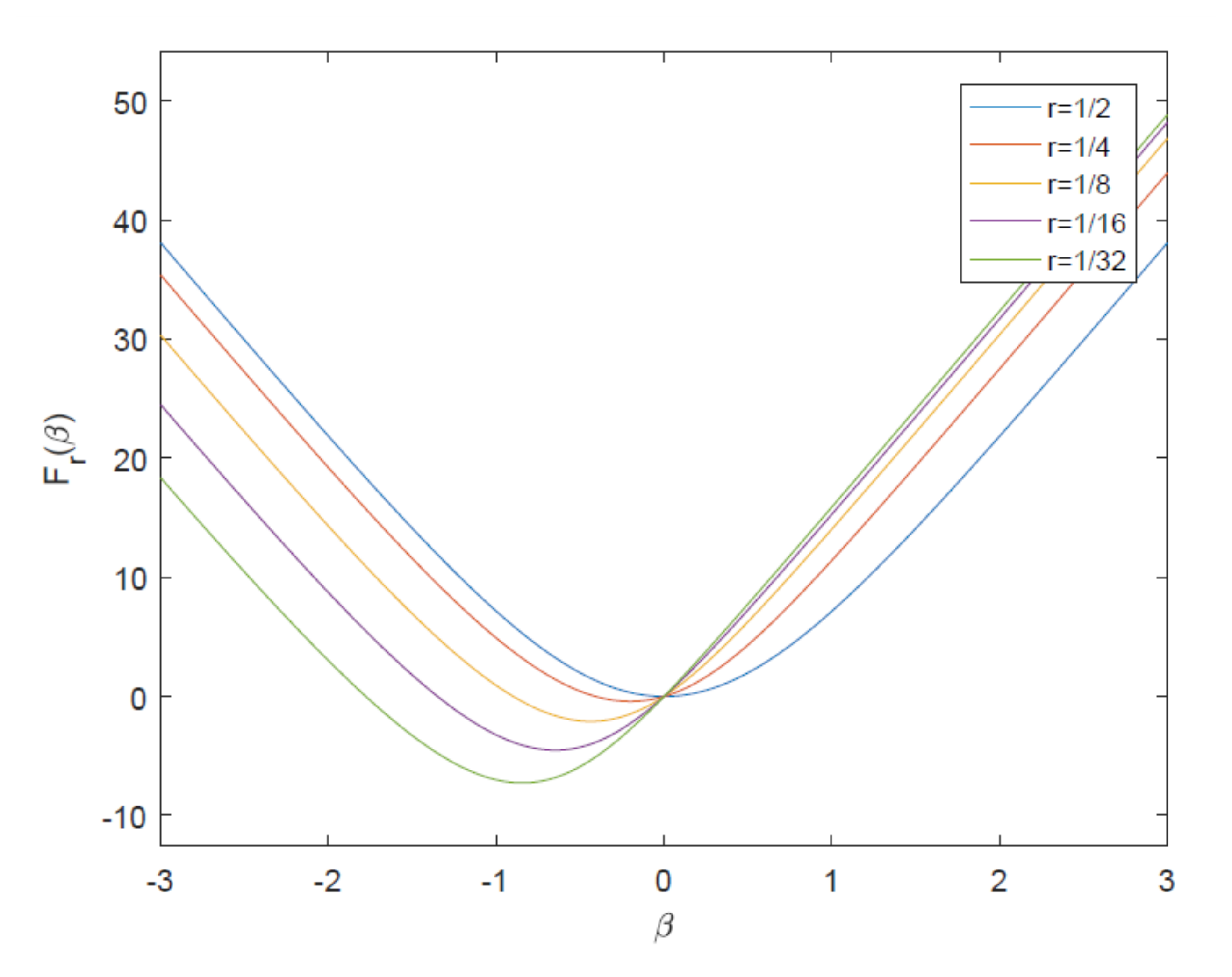}
		\caption{Plot of the $F_r(\beta)$ with ${\bf p}_i, 1\leq i \leq 5$ for different $r$ values.}\label{fig 2}
	\end{figure}
	Figure \ref{fig 2} shows the free energy behaviour for the multidimensional case $d = 2$ with $j=1$, ${\bf p}_1=(2,3), {\bf p}_2=(3,5), {\bf p}_3=(5,7), {\bf p}_4=(7,11), {\bf p}_5=(11,2)$, $S^{(1)}=S^{(2)}=\langle 2,3,5,7,11 \rangle$, $\gamma(S^{(1)})^{-1}=\gamma(S^{(2)})^{-1}=\frac{16}{77}$ and $C=(1-\frac{1}{6})(1-\frac{1}{15})(1-\frac{1}{35})(1-\frac{1}{77})(1-\frac{1}{22})\frac{77}{16}=\frac{2261}{660}$.
\end{example}

\bibliographystyle{amsplain}

\begin{thebibliography}{10}
	\bibitem{ban2021entropy}
	J.~C. Ban, W.~G. Hu, and G.~Y. Lai, \emph{On the entropy of multidimensional
		multiplicative integer subshifts}, Journal of Statistical Physics
	\textbf{182} (2021), no.~2, 1--20.
	
	\bibitem{ban2021LDP}
	J.~C. Ban, W.~G. Hu, and G.~Y. Lai, \emph{Large Deviation Principle of Multidimensional Multiple Averages on $\mathbb{N}^d$}, Indagationes Mathematicae (2021), to appear.
	
	\bibitem{ban2019pattern}
	J.~C. Ban, W.~G. Hu, and S.~S. Lin, \emph{Pattern generation problems arising
		in multiplicative integer systems}, Ergodic Theory and Dynamical Systems
	\textbf{39} (2019), no.~5, 1234--1260.
	
	\bibitem{RJB-1982}
	R.~J. Baxter, \emph{Exactly solved models in statistical mechanics}, Academic
	Press, 1982.
	
	\bibitem{bourgain1990double}
	J.~Bourgain, \emph{Double recurrence and almost sure convergence},
	Journal f\"{u}r die reine und angewandte Mathematik
	\textbf{404} (1990), 140--161.
	
	
	\bibitem{brunet2020dimensions}
	G.~Brunet, \emph{Dimensions of ``self-affine sponges" invariant under the action
		of multiplicative integers}, arXiv preprint arXiv:2010.03230 (2020).
	
	\bibitem{carinci2012nonconventional}
	G.~Carinci, J.~R. Chazottes, C.~Giardina, and F.~Redig, \emph{Nonconventional
		averages along arithmetic progressions and lattice spin systems},
	Indagationes Mathematicae \textbf{23} (2012), no.~3, 589--602.
	
	\bibitem{chazottes2014thermodynamic}
	J.~R. Chazottes and F.~Redig, \emph{Thermodynamic formalism and large
		deviations for multiplication-invariant potentials on lattice spin systems},
	Electronic Journal of Probability \textbf{19} (2014).
	
	\bibitem{dembo2009ldp}
	A.~Dembo and O.~Zeitouni, \emph{LDP for finite dimensional
		spaces}, Large deviations techniques and applications, Springer, 2009,
	pp.~11--70.
	
	\bibitem{ellis2006entropy}
	R.~S. Ellis, \emph{Entropy, large deviations, and statistical mechanics}, vol.
	1431, Taylor \& Francis, 2006.
	
	\bibitem{fan2014some}
	A.~H. Fan, \emph{Some aspects of multifractal analysis}, Geometry and Analysis
	of Fractals, Springer, 2014, pp.~115--145.
	
	\bibitem{fan2021multifractal}
	\bysame, \emph{Multifractal analysis of weighted ergodic averages}, Advances in
	Mathematics \textbf{377} (2021), 107488.
	
	\bibitem{fan2012level}
	A.~H. Fan, L.~M. Liao, and J.~H. Ma, \emph{Level sets of multiple ergodic
		averages}, Monatshefte f{\"u}r Mathematik \textbf{168} (2012), no.~1, 17--26.
	
	\bibitem{fan2016multifractal}
	A.~H. Fan, J.~Schmeling, and M.~Wu, \emph{Multifractal analysis of some
		multiple ergodic averages}, Advances in Mathematics \textbf{295} (2016),
	271--333.
	
	\bibitem{furstenberg1982ergodic}
	H.~Furstenberg, Y.~Katznelson, and D.~Ornstein, \emph{The ergodic theoretical
		proof of Szemer\'{e}di’s theorem}, Bulletin of the American Mathematical
	Society \textbf{7} (1982), no.~3, 527--552.
	
	\bibitem{georgii2011gibbs}
	H.~O. Georgii, \emph{Gibbs measures and phase transitions}, vol.~9, Walter de
	Gruyter, 2011.
	
	\bibitem{host2005nonconventional}
	B.~Host and B.~Kra, \emph{Nonconventional ergodic averages and nilmanifolds},
	Annals of Mathematics (2005), 397--488.
	
	\bibitem{kenyon2012hausdorff}
	R.~Kenyon, Y.~Peres, and B.~Solomyak, \emph{Hausdorff dimension for fractals invariant under multiplicative integers}, Ergodic Theory and Dynamical
	Systems \textbf{32} (2012), no.~5, 1567--1584.
	
	\bibitem{frantzikinakis2011some}
	N.~Frantzikinakis, \emph{Some open problems on multiple ergodic averages}, arXiv preprint arXiv:1103.3808 (2011).
	
	\bibitem{kifer2010nonconventional}
	Y.~Kifer, \emph{Nonconventional limit theorems}, Probability theory and related fields \textbf{148} (2010), no. 1-2, 71-106.
	
	\bibitem{kifer2014nonconventional}
	Y.~Kifer and S.~R.~S.~Varadhan, \emph{Nonconventional limit theorems in discrete and continuous time via martingales}, The Annals of Probability
	\textbf{42} (2014), no. 2, 649-688.
	
	
	\bibitem{peres2014dimensions}
    Y.~Peres, J.~Schmeling, S.~Seuret, and B.~Solomyak, \emph{Dimensions of some fractals defined via the semigroup generated by 2 and 3}, Israel Journal of Mathematics \textbf{199} (2014), no.~2, 687--709.
	
	\bibitem{peres2012dimension}
	Y.~Peres and B.~Solomyak, \emph{Dimension spectrum for a nonconventional ergodic average}, Real Analysis Exchange \textbf{37} (2012), no.~2, 375--388.
	
	\bibitem{pollicott2017nonlinear}
	M.~Pollicott, \emph{A nonlinear transfer operator theorem}, Journal of statistical physics \textbf{166} (2017), no.~3-4, 516--524.
	
\end{thebibliography}


\end{document}